\documentclass{article}
\usepackage{amsmath}

\usepackage{amsthm}
\usepackage{mleftright}

\usepackage{amssymb}
\usepackage{youngtab}

\usepackage{bbm}
\newtheorem{theorem}{Theorem}[section]
\newtheorem{proposition}[theorem]{Proposition}

\newtheorem{corollary}[theorem]{Corollary}
\newtheorem{conjecture}[theorem]{Conjecture}

\theoremstyle{definition}

\newtheorem{remark}[theorem]{Remark}

\newtheorem{definition}[theorem]{Definition}
\newtheorem{question}[theorem]{Question}

\usepackage{hyperref}
\usepackage{url}
\usepackage{float}

\usepackage{wasysym}
\usepackage[usenames,dvipsnames]{color}
\usepackage{tikz}
\usetikzlibrary{decorations.markings}
\usetikzlibrary{arrows}
\usetikzlibrary{arrows.meta}
\usetikzlibrary{shapes.geometric}
\usetikzlibrary{patterns}
\tikzstyle{empty}=[circle,draw=black!80,thick]
\tikzstyle{emptyn}=[circle,draw=black!80,fill=white,scale=0.5] 
\tikzstyle{nero}=[circle,draw=black!80,fill=black!80,thick] 
\usetikzlibrary{positioning, fit, calc}

\usepackage{caption,subcaption}
\DeclareCaptionStyle{mystyle}
{format=plain,%
	textformat=period,
	justification=RaggedRight,
	singlelinecheck=true,
}

\DeclareCaptionStyle{singlelineraggedright}
[justification=RaggedRight]
{style=mystyle}

\DeclareCaptionStyle{singlelinecentred}
[justification=Centering]
{style=mystyle}

\DeclareCaptionStyle{singlelineraggedleft}
[justification=RaggedLeft]
{style=mystyle}

\usepackage{bbm}
\usepackage{geometry}
\usepackage[font=footnotesize]{caption}

\def\e{\epsilon}

\def\so3{\mathrm{SO}(3)}
\def\fso3{\mathrm{FSO}(3)}

\title{A note on the Brown--Erd\H os--S\'os conjecture in groups}
\author{Jason Long}

\begin{document}

\maketitle
\begin{abstract}
We show that a dense subset of a sufficiently large group multiplication table contains either a large part of the addition table of the integers modulo some $k$, or the entire multiplication table of a certain large abelian group, as a subgrid. As a consequence, we show that triples systems coming from a finite group contain configurations with $t$ triples spanning $\mathcal{O}(\sqrt{t})$ vertices, which is the best possible up to the implied constant. We confirm that for all $t$ we can find a collection of $t$ triples spanning at most $t+3$ vertices, resolving the Brown--Erd\H os--S\'os conjecture in this context. The proof applies well-known arithmetic results including the multidimensional versions of Szemer\'edi's theorem and the density Hales--Jewett theorem.

This result was discovered simultaneously and independently by Nenadov, Sudakov and Tyomkyn~\cite{NST}, and a weaker result avoiding the arithmetic machinery was obtained independently by Wong~\cite{W}.
\end{abstract}

\section{Introduction}

A central open problem in extremal combinatorics is the Brown--Erd\H os--S\'os conjecture~\cite{BES}. We say that a subgraph $H'$ of a hypergraph $H$ is an $(r,s)$-configuration if $|E(H')|=s$ and $|V(H')|\le r$. The Brown--Erd\H os--S\'os conjecture states that, for any fixed positive integer $t\ge 3$, any 3-uniform hypergraph $H$ on $n$ vertices which does not contain a $(t+3,t)$-configuration has at most $o(n^2)$ edges. The number $t+3$ cannot be decreased, since random constructions can achieve $\Omega(n^2)$ edges while avoiding any $(t+2,t)$-configurations~\cite{BES}. The conjecture can be generalised to higher uniformity, but we shall focus on the 3-uniform case in this note.

Since its formulation in 1973 there has been a great deal of work on this problem. Ruzsa and Szemer\'edi~\cite{RS} resolved the first non-trivial case ($t=3$), but the conjecture remains open for all $t>3$. The strongest result to date is due to S\'ark\"ozy and Selkow~\cite{SS}, who showed that any 3-uniform hypergraph which does not contain a $(t+2+\lfloor \log_2 t\rfloor,t)$-configuration has at most $o(n^2)$ edges.

When tackling the Brown--Erd\H os--S\'os conjecture, we may additionally assume that the hypergraph $H$ is linear (as noted in~\cite{S}, for example). It is also clear that we may assume that $H$ is tripartite, since, given a 3-graph $H$, we may obtain a tripartite 3-graph $H'$ by taking three copies of the vertex set of $H$ and placing edges between these partitions corresponding to the edges of $H$. 

Given a linear, tripartite, 3-uniform hypergraph $H$ on $n+n+n$ vertices we can associate a partially labelled $n\times n$ grid by labelling position $(a,b)$ with label $c$ if $(a,b,c)\in E(H)$. Thus the Brown--Erd\H os--S\'os conjecture can be formulated in terms of a quasigroup -- this is noted in~\cite{S} and~\cite{SW}, for example. 

\begin{conjecture}[Brown--Erd\H os--S\'os]\label{BES}
	Fix $t\in\mathbb{Z}^+$ and $\e>0$. Then there exists $N=N(t,\e)$ such that for any quasigroup $G$ of order $n>N$ and any subset $A$ of the multiplication table of $G$ of density at least $\e$, we can find a $(t+3,t)$-configuration in $A$; that is to say, a set of $t$ triples in $A$ spanning at most $t+3$ vertices (i.e. rows, columns or labels).
\end{conjecture}

In light of this formulation, it is natural to ask the same question when $G$ is in fact a group, as the additional structure might provide greater local density than can be found in random constructions.

\begin{conjecture}[Brown--Erd\H os--S\'os for groups]\label{BESG}
	Fix $t\in\mathbb{Z}^+$ and $\e>0$. Then for any sufficiently large group $G$ and any subset $A$ of the multiplication table of $G$ of density at least $\e$, we can find a $(t+3,t)$-configuration in $A$.
\end{conjecture}

Since the Brown--Erd\H os--S\'os conjecture is resolved for $t\le 3$, the first interesting case of Conjecture~\ref{BESG} is $t=4$. In 2015, Solymosi~\cite{S} resolved this case, showing that Conjecture~\ref{BESG} holds for $t=4$.

Recently, Solymosi and Wong~\cite{SW} showed that much more is true, proving that the Brown--Erd\H os--S\'os threshold of $t+3$ vertices can in fact be surpassed in the groups setting. In particular, they prove that dense subsets of sufficiently large group multiplication tables contain sets of $t$ triples in $A$ spanning asymptotically only $3t/4$ vertices. Since their result concentrates on the case of large $t$, they do not match Conjecture~\ref{BESG} for small $t$ but prove that it holds for infinitely many $t$.

Given that the Brown--Erd\H os--S\'os threshold can be surpassed in the groups setting, one may ask what the correct behaviour should be in this case. Since $A$ corresponds to a linear hypergraph, we cannot find sets of $t$ triples in $A$ spanning fewer that $\sqrt{t}$ vertices, but can we approach this lower bound?

\begin{question}\label{q1}
	Let $t$ be a fixed positive integer. What is the smallest number $F(t)$ such that we are guaranteed to find an $(F(t),t)$-configuration in a dense subset of a sufficiently large group multiplication table? 	
\end{question}

In this note we answer this question up to a constant factor, and resolve Conjecture~\ref{BESG}. By applying machinery from arithmetic combinatorics, including the multidimensional Szemer\'edi theorem and a multidimensional variant of the density Hales--Jewett theorem, we prove that any dense subset of a sufficiently large group multiplication table contains a large subgrid belonging to one of two families: either the subgrid matches part of the multiplication table of a cyclic group, or the subgrid matches the entire multiplication table of $\mathbb{F}_p^m$ for some small prime $p$ and large $m$. A precise statement appears in Theorem~\ref{thm} following some notation. 

This reduces Question~\ref{q1} to a discrete optimisation problem, in which we must find configurations with $t$ edges spanning few vertices in each of the two cases resulting from our main theorem. We tackle this optimisation problem in Section~\ref{opt}, showing that $F(t)=\mathcal{O}(\sqrt{t})$ and resolving Conjecture~\ref{BESG} for all $t$.

\section{Notation and Statements}\label{stmt}

We write $\mathbb{Z}_n$ for the group of integers modulo $n$ under addition and we write $[k]$ for the set $\{0,1,\dots,k-1\}$. We begin with some definitions.

\begin{definition}
	By the \emph{multiplication table} of a group $G=(G,\circ)$ we mean the collection of triples $(a,b,a\circ b)$ for $a,b\in G$. The \emph{vertex set} will be given by three disjoint copies of $G$ called the \emph{row vertices}, \emph{column vertices} and \emph{label vertices}. We shall refer to the triples as the edges or \emph{faces} of the corresponding tripartite 3-uniform hypergraph. Typically, we will represent this as a labelled grid, with entry $(a,b)$ given label $a\circ b$. In the case that $G=(G,+)$ is an abelian group, we will usually call the multiplication table an \emph{addition table}.
\end{definition}

\begin{definition}
	By a \emph{subgrid} of a labelled grid, we mean the labelled grid contained in the intersection of some subset of the rows and columns.
\end{definition}

\begin{definition}
	We say that a labelled grid $A$ is \emph{isomorphic} to another labelled grid $B$ if we can biject the row sets, column sets and label sets of $A$ and $B$ in such a way that the resulting map is a graph isomorphism between the corresponding 3-graphs.
\end{definition}

Using this notation we reformulate Question~\ref{q1} in a precise way.

\begin{question}\label{q}
	Let $t$ be a fixed positive integer and $\epsilon>0$. Let $F(t)$ be minimal such that, given any subset $A$ of density at least $\epsilon$ of a sufficiently large (in terms of $t$ and $\epsilon$) group multiplication table, we may find an $(F(t),t)$-configuration in $A$. How does $F(t)$ grow with $t$? Is $F(t)\le t+3$ for all $t$?
\end{question}

In order to answer this question, we prove the following structural result.

\begin{theorem}\label{thm}
	Fix $k,m\in\mathbb{Z}^+$ and $\epsilon>0$. Then there exists $N=N(k,m,\epsilon)$ such that, for any group $G$ of order $n>N$ and any subset $A$ of the multiplication table of $G$ of density at least $\epsilon$, $A$ contains either a subgrid isomorphic to the addition table of $[k]$ as a subset of $\mathbb{Z}_K$ for some $K\ge k$, or a subgrid isomorphic to the addition table of $\mathbb{Z}_p^m$ for some $p<k$ prime.
\end{theorem}
\begin{remark}\label{rem}
	This result is `best possible' in terms of finding configurations with many edges spanned by few vertices, since if $A$ is simply taken to be the addition table of $[n/2]$ as a subset of $\mathbb{Z}_n$, say, then any subgrid of $A$ is isomorphic to part of a larger addition table and we cannot improve on the first case of the theorem. Similarly, if $A$ is simply the addition table of $\mathbb{Z}_p^t$ for small $p$ and large $t$ then we cannot improve on the second case.
\end{remark}

\section{Proof of Theorem~\ref{thm}}

We start by introducing the arithmetic machinery that we use later. We begin with a multidimensional version of Szemer\'edi's theorem~\cite{multiS}.

\begin{theorem}[Multidimensional Szemer\'edi Theorem]\label{mST}
	Let $k,t\in\mathbb{Z}^+ $ and let $\epsilon>0$. Then there exists $N=N(\epsilon, k, t)$ such that for any $n>N$ and any $A\subset \mathbb{Z}_n^t$	of density at least $\epsilon$, we can find $a_1,a_2,\dots,a_t,d\in \mathbb{Z}_n$ such that
	$$(a_1+i_1d,a_2+i_2d,\dots,a_t+i_td)\in A$$
	for each $i_j\in \{0,\dots,k-1\}$. In other words, $A$ contains the Cartesian product of $t$ arithmetic progressions of length $k$ with the same common difference.
\end{theorem}

We shall also need a multidimensional version of the density Hales--Jewett theorem~\cite{dHalesJewett}. We recall the definition of a combinatorial line.

\begin{definition}
	A \emph{combinatorial line} in $\mathbb{Z}_m^n$ is a set $U$ of the form $$U=\{(x_1,\dots,x_n) \,| \, x_i\text{ constant on } I, \, x_j=z_j\text{ for }j\not\in I\}$$
	for some indexing set $I\subset \{1,\dots,n\}$ and some $z\in \mathbb{Z}_m^n$. A \emph{combinatorial subspace of dimension $k$} is a set $U$ of the form $$U=\{(x_1,\dots,x_n) \,| \, x_i\text{ constant on each } I_s, \, x_j=z_j\text{ for }j\not\in \cup_s I_s\}$$
		for some collection of $k$ disjoint indexing sets $I_s\subset \{1,\dots,n\}$, and some $z\in \mathbb{Z}_m^n$.
\end{definition}

\begin{theorem}[Density Hales--Jewett]\label{DHJ}
	Fix $m\in\mathbb{Z}^+$ and let $\epsilon>0$. Then there exists $N=N(\epsilon, m)$ such that for any $n>N$ and any $A\subset \mathbb{Z}_m^n$	of density at least $\epsilon$, we can find a combinatorial line inside $A$.
\end{theorem}

The density Hales--Jewett theorem easily implies its own multidimensional variant -- for a proof, see~\cite{DKT} for example.

\begin{corollary}[Multidimensional density Hales--Jewett]\label{DHJ2}
	Let $m,k$ be fixed positive integers and let $\epsilon>0$. There exists $N=N(\epsilon, m,k)$ such that for any $n>N$ and any $A\subset \mathbb{Z}_m^n$	of density at least $\epsilon$, we can find an entire combinatorial subspace of dimension $k$ inside $A$.
\end{corollary}

We will need a further variant of density Hales--Jewett, which follows easily from Corollary~\ref{DHJ2} by applying the same idea used to extend from Theorem~\ref{DHJ} to Corollary~\ref{DHJ2}.

\begin{corollary}\label{mDHJ}
	Let $k,t$ be fixed positive integers, $p$ a fixed prime, and let $\epsilon>0$. There exists $N=N(\epsilon, p, k, t)$ such that for any $n>N$ and any $A\subset (\mathbb{Z}_p^n)^t$ of density at least $\epsilon$, we can find a subspace $\Gamma$ of dimension $k$ and $a_1,\dots,a_t\in \mathbb{Z}_p^n$ such that 
	$$(a_1+\Gamma)\times (a_2+\Gamma) \times \dots \times (a_t+\Gamma) \subset A.$$
\end{corollary}
\begin{proof}	
	We simply identify $(\mathbb{Z}_p^n)^t$ with $\mathbb{Z}_{p^t}^n$ in the obvious way. We can then apply Corollary~\ref{DHJ2} to find a combinatorial subspace of dimension $k$ inside $A$, which gives us an affine subspace of dimension $k$. The result follows by translating back to $(\mathbb{Z}_p^n)^t$.
\end{proof}

Lastly, we will need Pyber's theorem~\cite{Pyber} which provides us with a large abelian subgroup of $G$. 

\begin{theorem}[Pyber's Theorem]\label{Pyber}
	There is a universal constant $c>0$ such that any group $G$ of order $n$ contains an abelian subgroup of order at least $e^{c\sqrt{\log(n)}}$.
\end{theorem}

We are now ready to prove Theorem~\ref{thm}.

\begin{proof}[Proof of Theorem~\ref{thm}]
	We begin by applying Theorem~\ref{Pyber}, which states that $G$ contains an abelian subgroup $G'$ of order at least $\exp(c\sqrt{N})$ for some absolute constant $c>0$. In particular, $N'=|G'|$ tends to infinity with $N$.
	
	Note that the multiplication table of $G$ can be partitioned into the Cartesian products of left cosets of $G'$ with right cosets of $G'$. Since $A$ has density at least $\epsilon$ in the full multiplication table $G\times G$, we know that there exists $r,s\in G$ such that $A$ has density at least $\epsilon$ in the Cartesian product $rG'\times G's$. The part of the multiplication table corresponding to this Cartesian product is isomorphic to the addition table of $G'$. Let $A'=A\cap(rG'\times G's)$ be the subset of $rG'\times G's$ of density at least $\epsilon$ obtained from $A$. Note that $G'$ is a finite abelian group, and can therefore be written as a direct product of cyclic groups of prime power order.
	
	Suppose that $G'$ has a cyclic factor $\mathbb{Z}_T$. Then, as above, we can find a subset $A''$ which has density at least $\epsilon$ in a Cartesian product of two cosets of $\mathbb{Z}_T$ in $G$, and this Cartesian product is isomorphic to the addition table of $\mathbb{Z}_T$. Thus $A''$ corresponds to a subset of the $T\times T$ addition table of density at least $\epsilon$. By Theorem~\ref{mST}, if $T>T(k,\epsilon)$ is sufficiently large then we can find a Cartesian product of two arithmetic progressions $(a,a+d,\dots,a+(k-1)d)$ and $(b,b+d,\dots,b+(k-1)d)$ in $A''$. The labels in this subgrid belong to the set $\{a+b,a+b+d,\dots,a+b+2d\}$. Indeed, this subgrid is isomorphic to the addition table $\{0,\dots,k-1\}\times\{0,\dots,k-1\}\subset \mathbb{Z}^2$ and so we are in the first case of the statement of the theorem.
	
	So we are done if $G'$ contains a cyclic factor $\mathbb{Z}_T$ with $T>T(k,\epsilon)$. Therefore we may assume that all factors of $G'$ are cyclic groups with bounded (prime power) order. Since $|G'|$ tends to infinity with $N$, we see that for any positive integer $M$, if $N$ is sufficiently large we may find (by the pigeonhole principle) a cyclic factor $\mathbb{Z}_{p^a}$ which appears to the power $M$. In particular, $G'$ contains $\mathbb{Z}_p^M$ as a subgroup.
	
	As above, we note that this means that we may find $A''\subset A$ which has density at least $\epsilon$ in the Cartesian product of two cosets of $\mathbb{Z}_p^M$ inside $G$, and this product is isomorphic to the multiplication table of $\mathbb{Z}_p^M$. If $M$ is sufficiently large in terms of $m$, then by Corollary~\ref{mDHJ} we can find the complete Cartesian product of $a+\mathbb{Z}_p^m$ and $b+\mathbb{Z}_p^m$ inside $A''$. This complete Cartesian product is isomorphic to the addition table of $\mathbb{Z}_p^m$. If $p\ge k$ then we can find the addition table of $\mathbb{Z}_p$ and we are in the first case of the theorem, and otherwise we have $p<k$ and are in the second case.
\end{proof}

We now see how Theorem~\ref{thm} simplifies Question~\ref{q}. We let $f(t)$ be minimal such that we can find an $(f(t),t)$-configuration in the addition table of $[k]\subset\mathbb{Z}_K$ for any $K\ge k$ sufficiently large compared to $t$. Similarly, for each prime $p$ we let $g_p(t)$ be minimal such that we can find an $(g_p(t),t)$-configuration in the addition table of $\mathbb{Z}_p^m$ for any sufficiently large $m$ (in terms of $t$ and $p$).

\begin{corollary}\label{cormain}
	We have that $F(t)=\max_{p}(f(t),g_p(t))$.
\end{corollary}
\begin{proof}
	Clearly $F(t)\le\max_{p}(f(t),g_p(t))$. 
	
	For the other direction, we apply Theorem~\ref{thm} for choices of $k$ and $m$ sufficiently large in terms of $t$. Given a subset $A$ of density at least $\epsilon$ of the multiplication table of some sufficiently large group $G$, we may therefore find a subgrid isomorphic to the entire addition table of $[k]$ as a subset of $\mathbb{Z}_K$ for some $K\ge k$, or a subgrid isomorphic to the entire addition table of $\mathbb{Z}_p^m$ for some $p<k$ prime. If $k$ and $m$ are chosen large enough (in terms of $t$ only), we deduce that $A$ contains either an $(f(t),t)$-configuration or a $(g_p(t),t)$-configuration and so $F(t)\ge \max_{p}(f(t),g_p(t))$.
\end{proof}

We have thus reduced Question~\ref{q} to the problem of finding $f(t)$ and $g_p(t)$. We will devote the next section to tackling this discrete optimisation problem; providing an exact, closed form answer for all $t$ is tricky because of certain divisibility considerations. 
\section{Finding locally dense configurations}\label{opt}

In order to keep the note brief, we will not attempt to give the best possible bounds. We will instead show that $F(t)=\mathcal{O}(\sqrt{t})$, and, because of the connection with Conjecture~\ref{BES}, we will separately confirm that $F(t)\le t+3$ for all $t$.

For the analysis of the discrete optimisation problem arising from Corollary~\ref{cormain}, it simplifies the calculations to try and maximise the number of faces induced by a fixed number $v$ of vertices rather than minimise the number of vertices spanned by a fixed number $t$ of faces. Thus we let $f'(v)$ be the maximal number of faces that can be spanned by a set of $v$ vertices in the addition table of $[k]\subset\mathbb{Z}_K$ for any $K\ge k$ sufficiently large compared to $v$, and observe that if $f'(v)\ge t$ then $f(t)\le v$. Similarly, we let $g_p'(v)$ be the maximal number of faces that can be spanned by a set of $v$ vertices in the addition table of $\mathbb{Z}_p^m$ for any $m$ sufficiently large in terms of $v$ and $p$, and observe that if $g_p'(v)\ge t$ then $g_p(t)\le v$.

\begin{proposition}\label{f}
	We have that $f'(v)\ge (1+o(1))v^2/12$, and therefore $f(t)\le (\sqrt{12}+o(1))\sqrt{t}$.
\end{proposition}
\begin{proof}
	We work in the addition table of $[k]\subset\mathbb{Z}_K$ for $K\ge k\ge v$. Given $r$ rows and $r$ columns, we can optimise the density of our configuration by including the $s$ most numerous labels. The labels in the addition table are constant along falling diagonals. In the worst case, each falling diagonal corresponds to  a different label, in which case the most numerous label occurs $r$ times, the next two most numerous labels occur $r-1$ times each, etc. Therefore, by including the $s$ most numerous labels, we include a total of at least
	$$r+(r-1)+(r-1)+(r-2)+(r-2)+\dots+(r-\lceil (s-1)/2 \rceil)$$
	$$=sr-s(s-1)/4-\frac12\lfloor s/2\rfloor$$
	different faces. The total number of vertices is $2r+s$ so we seek to maximise this expression with respect to the constraint that $2r+s\le v$. Taking $r=\lfloor v/3\rfloor$ and $s=\lceil v/3\rceil$, and noting that $f'(v)$ is an increasing function of $v$, the proposition follows.
\end{proof}

\begin{proposition}\label{g}
	We have that $g_p'(v)\ge (1+o(1))v^2/49$ for all $p$, and therefore $g_p(t)\le (7+o(1))\sqrt{t}$.
\end{proposition}
\begin{proof}
	We work in the addition table $T$ of $\mathbb{Z}_p^m$ for $m$ large. If $p\ge v/3$ then the construction in the proof of Proposition~\ref{f} finds a configuration in the addition table of $\mathbb{Z}_p$ with $(1+o(1))v^2/12$ faces and so we are done.
	
	Otherwise, let $l$ be minimal such that $3p^{l+1}> v$. For $m$ sufficiently large, $T$ contains a subgrid isomorphic to the multiplication table of $\mathbb{Z}_p^{l+1}$.  We can partition this multiplication table into the Cartesian products of the cosets of $\mathbb{Z}_p^{l}$. These Cartesian products can be arranged into a $p\times p$ grid of blocks ($p^l\times p^l$ subgrids) corresponding to entries of the addition table $\mathbb{Z}_p\times \mathbb{Z}_p$. 
	
	We form our configuration by taking a union of these blocks. Let $v=\lambda p^l$, and so $\lambda\in [3,3p).$ The number $B$ of blocks that we can use is precisely the maximum number of faces induced by $\lfloor \lambda \rfloor$ vertices in the addition table of $\mathbb{Z}_p$. The number of vertices in the resulting configuration will be at most $v$, and the number of edges will $Bp^{2l}=Bv^2/\lambda^2$.
	
	Since $p> \lambda/3$ we could use the construction idea from Proposition~\ref{f}. Unfortunately, we cannot assume that $\lambda$ is large (in which case we could take approximately $\lambda^2/12$ blocks and therefore approximately $v^2/12$ faces) and the worst cases for this construction will in fact be decided by the best options for small $\lambda$. 
	
	In order to minimise the calculation, we will instead simply take an $a\times a$ grid of these blocks, and we shall choose $a$ maximal subject to our constraint on the number of vertices. 
	
	If we take the bottom left $a\times a$ grid of these Cartesian products we obtain a configuration with $ap^l$ rows, $ap^l$ columns and at most $(2a-1)p^l$ labels. The configuration has $a^2p^{2l}$ faces. Taking $a$ maximal so that $4a-1 \le v/p^l=\lambda$, we obtain a configuration $C$ with at most $v$ vertices. 
		
	By the maximality of $a$ we see that $a=\lfloor \lambda/4+1/4\rfloor$ so in particular $a\ge \max(1,\lambda/4-3/4)$. The number of faces of the configuration $C$ is $a^2p^{2l}$ which is therefore at least 
	$$\max\bigg(\frac{v^2}{\lambda^2}, \frac{(\lambda-3)^2}{16\lambda^2}v^2\bigg)$$
	which takes its minimal value of $v^2/49$ when $\lambda=7$.	
%
%
\end{proof}

\begin{remark}
	It is not hard to show that Proposition~\ref{f} is in fact best possible, and $1/12$ is the correct constant in the limit. On the other hand, Proposition~\ref{g} does not give the correct constant. As mentioned in the proof, combining the construction in Proposition~\ref{f} with a careful analysis of small $\lambda$ would allow improvements to be made quite easily. We can also make use of leftover vertices (when $\lambda$ is not an integer, a union of blocks uses only $\lfloor\lambda\rfloor p^l<v$ vertices, leaving some  unused) to interpolate between the constructions for integer values of $\lambda$. Using these techniques we can improve the constant from $1/49$ to $5/64$. However, the calculations are quite involved and the result would still not be the best possible, so we have tried to find a compromise between giving the best bounds that we can and providing a streamlined result. 
\end{remark}

Combining Propositions~\ref{f} and~\ref{g} with Corollary~\ref{cormain} gives the following result.

\begin{corollary}\label{cor}
$F(t)=\mathcal{O}(\sqrt{t})$ (in fact, $F(t)\le (7+o(1))\sqrt{t}$).
\end{corollary}

Therefore, the Brown--Erd\H os--S\'os threshold of $(t+3,t)$ is far below what can be found given the extra group structure. Nevertheless, we will now confirm that we do indeed prove the Brown--Erd\H os--S\'os conjecture in the context of group multiplication tables, which essentially involves checking that sufficiently dense configurations exist for the small values of $t$, as well as for large $t$ as verified by Corollary~\ref{cor}.

\begin{proposition}\label{prop}
	We have that $F(t)\le t+3$ for all $t\ge 3$.
\end{proposition}
\begin{proof}
	Although much better bounds than $t+3$ are possible for large $t$, it will be most convenient simply to find $(t+3,t)$-configurations in the addition table of $[k]\subset\mathbb{Z}_K$ for $K\ge k$  large, and also in the addition table of $\mathbb{Z}_p^m$ for $m$ large. The result will then follow by Corollary~\ref{cormain}.
	
	For the first case, working in the addition table of $[k]\subset\mathbb{Z}_K$, we note that taking the points in positions $(0,0)$, $(0,1)$, and $(1,0)$ gives the configuration
	$$\begin{matrix} 1&\\0&1\end{matrix}$$
	which has 6 vertices spanning 3 faces. Next, we can include the point in position $(1,1)$, which introduces one new vertex (a new label, $2$) and one new face. Then the point in position $(2,0)$ introduces one new vertex (a new column) and one new face, and then the point in position $(2,1)$ introduces one new vertex (a new label) and one new face. Continuing, we introduce the points in positions $(i,0)$ and $(i,1)$ for each $i$ until we have $t$ faces. At this point we have a configuration with $t$ faces spanning $t+3$ vertices.
	
	In the second case, we are working in the addition table of $\mathbb{Z}_p^m$ for $m$ large. We can use the above argument to find an $(r+3,r)$-configuration for $r$ up to $2p-1$ by taking the bottom two rows, minus the final face, of the multiplication table of some copy of $\mathbb{Z}_p$. When we add in the final point in position $(p-1,1)$ we re-use the label in position $(0,0)$ so we get an $(r+2,r)$-configuration. We can then start again in a new copy of $\mathbb{Z}_p$, including the corresponding points one by one in the same order as before. Our first point introduces two new vertices (a new row and new column) for just one more face, but since we are adding it to an $(r+2,r)$-configuration we get back to an $(r+3,r)$-configuration. Thereafter we add at most one new vertex with every new face. Once we finish the bottom two rows of the next copy of $\mathbb{Z}_p$ we can start again in another copy, and we can continue until we have $t$ faces. At that point we will span at most $t+3$ vertices.
\end{proof}

Conjecture~\ref{BESG} follows immediately from Proposition~\ref{prop}.

\section{Concluding remarks}

We have shown that the Brown--Erd\H{o}s--S\'os conjecture is true for hypergraphs with an underlying group structure, and in fact a much stronger result is possible. We give a bound of $\mathcal{O}(\sqrt{t})$ on the minimum size of a collection of vertices spanned by $t$ edges, which is tight up to the implied constant. Theorem~\ref{thm} provides an explanation for this local density by showing that bounded-size subgrids manifesting an abelian group structure can be found in any dense subset of a group multiplication table.

It is natural to wonder the ability to find many configurations with density beating the Brown--Erd\H{o}s--S\'os threshold is in some way connected to group-like structure. Are there interesting structural constraints weaker than the group axioms that still provide local density beyond the Brown--Erd\H{o}s--S\'os threshold? Or does the existence of many $(r,s)$-configurations with $r$ sufficiently small in terms of $s$ require an underlying group structure?

\section*{Acknowledgements}
The author would like to thank Tim Gowers for several helpful comments.

\end{document}